\documentclass[11pt,reqno]{amsart}
\usepackage{amssymb,latexsym}
\usepackage{graphicx}
\usepackage[usenames]{color}
\usepackage{MnSymbol}
\usepackage{url}
\usepackage{comment}		

\usepackage{algorithmicx}
\usepackage[ruled]{algorithm}
\usepackage[noend]{algpseudocode}

\algnewcommand\algorithmicinput{\textbf{Input:}}
\algnewcommand\Input{\item[\algorithmicinput]}%
\algnewcommand\algorithmicoutput{\textbf{Output:}}
\algnewcommand\Output{\item[\algorithmicoutput]}%

\newcommand{\TAU}{{\Xi}}

\calclayout

\makeatletter
\newcommand{\monthyear}[1]{%
  \def\@monthyear{\uppercase{#1}}}
\newcommand{\volnumber}[1]{%
  \def\@volnumber{\uppercase{#1}}}
\AtBeginDocument{%
\def\ps@plain{\ps@empty
  \def\@oddfoot{\@monthyear \hfil \thepage}%
  \def\@evenfoot{\thepage \hfil \@volnumber}}
\def\ps@firstpage{\ps@plain}
\def\ps@headings{\ps@empty
  \def\@evenhead{%
    \setTrue{runhead}%
    \def\thanks{\protect\thanks@warning}%
    \uppercase{The Fibonacci Quarterly}\hfil}%
  \def\@oddhead{%
    \setTrue{runhead}%
    \def\thanks{\protect\thanks@warning}%
    \hfill\uppercase{Three Cousins of Recam\'{a}n's Sequence}}%
  \let\@mkboth\markboth
  \def\@evenfoot{%
    \thepage \hfil \@volnumber}%
  \def\@oddfoot{%
    \@monthyear \hfil \thepage}%
  }%
\footskip=25pt
\pagestyle{headings}%
}
\makeatother

\newcommand{\beql}[1]{\begin{equation}\label{#1}} 
\newcommand{\eeq}{\end{equation}} 

\theoremstyle{plain}
\numberwithin{equation}{section}
\newtheorem{thm}{Theorem}[section]

\newtheorem{theorem}[thm]{Theorem}
\newtheorem{lemma}[thm]{Lemma}

\begin{document}
\monthyear{Month Year}
\volnumber{Volume, Number}
\setcounter{page}{1}

\title{Three Cousins of Recam\'{a}n's Sequence}
\author{Max A. Alekseyev}
\address{Department of Mathematics \& Computational Biology Institute\\
George Washington University\\ 
Ashburn, VA  20147, USA}
\email{maxal@gwu.edu}
\author{Joseph Samuel Myers}
\address{c/o Trinity College\\
Cambridge\\
CB2 1TQ, UK}
\email{jsm@polyomino.org.uk}
\author{Richard Schroeppel}
\address{500 South Maple Dr.\\
Woodland Hills\\
 UT 84653, USA}
\email{rcs@xmission.com}
\author{S. R. Shannon}
\address{P.O. Box 2260 \\
Rowville, Victoria 3178\\
AUSTRALIA}
\email{scott\_r\_shannon@hotmail.com}
\author{N. J. A. Sloane}
\address{The OEIS Foundation Inc. \\
11 South Adelaide Ave. \\
Highland Park, NJ 08904, USA}
\email{njasloane@gmail.com}
\author{Paul Zimmermann},
\address{INRIA Nancy -- Grand Est, LORIA \\
F-54600 Villers-l\`es-Nancy \\
FRANCE}
\email{Paul.Zimmermann@inria.fr}

\begin{abstract}
Although $10^{230}$ terms of Recam\'{a}n's sequence have been computed, it remains a mystery.  
Here three distant cousins of that sequence  are described, one of which is
also mysterious.
(i) $\{A(n), n \ge 3\}$ is defined as follows. Start with $n$, and add $n+1$, $n+2$, $n+3, \ldots$,
stopping after adding $n+k$ if the sum $n + (n+1) + \ldots + (n+k)$
is divisible by $n+k+1$. Then $A(n)=k$. We determine $A(n)$ and
show that $A(n) \le n^2 - 2n - 1$.
(ii)  $\{B(n), n \ge 1\}$ is a multiplicative analog of $\{A(n)\}$. Start with $n$,
and successively multiply by $n+1$,  $n+2, \ldots$, stopping after multiplying by $n+k$ if the product
$n(n+1)\cdots (n+k)$  is divisible by $n+k+1$. Then $B(n)=k$. 
We conjecture that $\log^2 B(n) = (\frac{1}{2}+o(1))  \log n \log\log n$.
(iii) The third sequence, $\{C(n), n \ge 1\}$, is the most interesting, because the most mysterious. 
Concatenate the decimal digits of $n, n+1, n+2, \ldots$ until the concatenation  $n \| n+1\| \ldots \| n+k$
is divisible by $n+k+1$. Then $C(n)=k$. If no such $k$ exists we set $C(n)=-1$.
We have found $k$ for all $n \le 1000$ except
for two cases. Some of the numbers involved are quite large.
For example, $C(92) = 218128159460$, and the concatenation $92\|93\|\ldots\|(92+C(92))$
is a number with about $2\cdot 10^{12}$ digits. We have only a probabilistic argument that such a $k$
exists for all $n$.
\end{abstract}

\maketitle


\section{Introduction}\label{Sec1}
Recam\'{a}n's sequence $\{R(n), n \ge 0\}$  is defined by $R(0) = 0$ and, for $n \ge 1$,  
$R(n) = R(n-1) - n$ if that number is positive and not already in the sequence,
and otherwise $R(n)=R(n-1)+n$ (in the latter case repeated terms are permitted).
Terms $R(0)$ through $R(11)$ are
$0, 1, 3, 6, 2, 7, 13, 20, 12, 21, 11, 22$.
The sequence was contributed  
by Bernardo Recam\'{a}n Santos in 1991  to what is now
 the {\em On-line Encyclopedia of Integer Sequences}  (or {\em OEIS}) \cite{OEIS}. 
The most basic question about this sequence is still unanswered: does every nonnegative integer appear?
The fifth author (NJAS) and several Bell Labs colleagues developed a method for speeding up the
computation of the sequence, and in 2001 
Allan Wilks used it to compute the first $10^{15}$ terms. At that point every number below 
$852655$ had appeared, but $852655 = 5 \cdot 31 \cdot 5501$ itself was missing.
Benjamin Chaffin  has continued this work, and in 2018  reached $10^{230}$ terms~\cite{CRSW}.
However, $852655$ is still missing.

Thirty years ago it seemed like a very plausible conjecture that every number would eventually 
appear in Recam\'{a}n's sequence.
Today, it is not so clear. For much more about this sequence, see 
entry  A005132\footnote{Six-digit numbers prefixed by A refer to entries in \cite{OEIS}.}
 in \cite{OEIS}.
 
A somewhat similar situation arose in connection with the third of our new sequences, $\{C(n)\}$,
discussed in Sect.~\ref{SecC}. We have no proof that the search for $C(n)$ will always terminate,
 and after reaching
$10^{11}$ in our search for $C(44)$, we were beginning to have doubts.
However, after considerably more computation using a different algorithm 
(described in~\S\ref{Sec4.1}-\S\ref{Sec4.3}) 
we were able to show that 
$C(44) = 2783191412912$.
Similar results for other hard-to-find values of $C(n)$ have convinced us that 
the search for $C(n)$ should always terminate.

In Recam\'{a}n's sequence we  start by trying to subtract $n$ from
the previous term. In the three sequences discussed here, to compute
$A(n)$, $B(n)$, or $C(n)$ we define an intermediate
sequence  which starts with $n$
and is extended by either  {\em adding}  ($A(n)$, Section~\ref{SecA}),
{\em multiplying by}  ($B(n)$, Section~\ref{SecB}),
or
{\em concatenating}  ($C(n)$, Section~\ref{SecC}) $n+i$ to the $i$th term to  get the next term.

\paragraph{\it Notation.} A centered dot ($\cdot$)  indicates multiplication.
 In Section~\ref{SecA}, $T_n$ denotes the triangular number $n(n+1)/2$;
 in Sections~\ref{SecB} and \ref{SecC} a vertical bar ($\mid$) means ``divides''
 and  $\nu_p(n)$ denotes the exponent of the highest power of $p$ that divides $n$ (the {\em $p$-adic valuation} of $n$);
 and in Section~\ref{SecC}, $\|$ denotes concatenation of 
 the  decimal (or more generally, base $b$) representations of numbers.
 Also in Section~\ref{SecC} we distinguish between the {\em number} $\alpha \bmod \gamma$
 and the {\em congruence} $\alpha \equiv \beta \pmod{\gamma}$.


\section{The additive version, $\{A(n)\}$.}\label{SecA}
To find $A(n)$, $n \ge 3$, we define an intermediate sequence $\{a_n(i), i \ge 0\}$ by starting with
$a_n(0)=n$, and, for $i \ge 1$, letting $a_n(i)=a_n(i-1)+n+i$. We stop when we
reach a term $a_n(k)$ which is divisible by $d=n+k+1$, and set $A(n)=k$.
In other words, if the number $d$ that we are about to add to $a_n(k)$ actually divides $a_n(k)$, 
then instead of adding it we stop.

An equivalent definition is that $A(n)$ is the smallest positive integer $k = k(n)$ such that $d(n) = n+k+1$ divides
\beql{EqA1}
a_n(k) ~=~ (k+1)n+\frac{k(k+1)}{2}\,.
\eeq

If $n=3$, for example, the sequence $\{a_3(i)\}$ is $a_3(0)=3$, $a_3(1)=7$, $a_3(2)=12$, and we stop
with $k=2=A(3)$  since $12$ is divisible by $d=3+2+1=6$. For $n=4$ the sequence $\{a_4(i)\}$ is
$4,9,15,22,30,39,49,60$, where we stop with $k=7=A(4)$ since $a_4(7)=60$ is divisible by $d=4+7+1=12$.

Table \ref{TabA} gives the values of $A(n)=k(n)$, $d(n)=n+k(n)+1$, $p(n) = a_n(k)$, and 
$q_n=p(n)/d(n)$ for
$n=3,4,\ldots,17$.
The last column gives the values of a parameter $m$ that will arise when we
 relate this problem to triples of triangular numbers $T_j$.
We start the table at $n=3$, because although we can certainly define
the sequence $\{a_2(i)\}$, it turns out that $a_2(i) = T_{i+2}-1$,
and it is easy to show that $T_{i+2}-1$ is never divisible by $i+3$. So $A(2)$ does not exist.

On the other hand, $A(n)$ exists for all $n \ge 3$.  The record high values of
$A(n)$ in the table at $n=3,4,5,8,17$ suggest that 
$(n-1)^2-2$ is an upper bound.
If we take $k=(n-1)^2-2$
for $n \ge 3$ we find from \eqref{EqA1} that $a_n(k) = (n+1)n(n-1)(n-2)/2$,
which is indeed divisible by $n+k+1 = n(n-1)$,
and so $A(n) \le (n-1)^2-2$.

\begin{table}[htb]
\caption{}\label{TabA}
$$
\begin{array}{rcccccc}
n & A(n)=k(n) & d(n) & p(n)=a_n(k) & q_n=p(n)/d(n)  & m \\
\hline
3  & 2          & 6 & 12           & 2              & 3    \\
4  & 7          & 12 & 60           & 5            & 6    \\
5  & 14          & 20 & 180        & 9           & 10    \\
6  & 3          & 10 & 30           & 3            & 6    \\
7  & 6          & 14 & 70           & 5            & 8    \\
8  & 47          & 56 & 1512       & 27        & 28    \\
9  & 14          & 24 & 240         & 10        & 13    \\
10 & 4          & 15 & 60           & 4           & 10    \\
11 & 10          & 22 & 176           & 8        & 13    \\
12 & 20          & 33 & 462           & 14      & 18    \\
13 & 25          & 39 & 663           & 17       & 21    \\
14 & 11          & 26 & 234           & 9         & 16    \\
15 & 5          & 21 & 105           & 5           & 15    \\
16 & 31          & 48 & 1008           & 21     & 26    \\
17 & 254          & 272 & 36720     & 135    & 136    \end{array}
$$
\end{table}

The sequences $\{A(n)\}$, $\{d(n)\}$, and $\{p(n)\}$ have now been added to \cite{OEIS}: $\{A(n)\}$ is
 A332542. However, to our surprise,
the $\{q_n\}$ sequence appeared to match an existing sequence, although
with a shift in subscripts. 
For $n \ge 2$, let $\TAU(n)$ denote the smallest $k>0$ such that
\beql{EqTT1}
T_n+T_k = T_m
\eeq
for some integer $m$. The initial values are $\TAU(2)=2$,
$\TAU(3)=5$, $\TAU(4)=9, \ldots$ (A082183)
and apparently  agree with $q_{n+1}$.
We will show in Theorem~\ref{ThAT} that this is true.

The representation of numbers as sums or differences of triangular numbers
is a classical subject, going back to Fermat and Gauss, and has been
studied in many recent papers
 \cite{BaWe79, Hag97, LeZa93, Nyb01, Sie63, 
 Tri08, Ula09, Vai72}.
However, we were unable to find Theorems~\ref{ThK} and \ref{ThAT} in the literature.

Following \cite{LeZa93} we define a
{\em triangular triple} to be an ordered triple of nonnegative integers $[n,k,m]$ satisfying \eqref{EqTT1}.
We say that a triple is {\em trivial} if any of $n, k, m$ are  zero. 

It is easy to see that $\TAU(n)$ exists, 
since it is straightforward to check that $[n, T_n-1, T_n]$ is a triangular 
triple for $n \ge 1$. 
So $\TAU(n) \le T_n-1$.

We will say exactly what all the triangular triples $[n,k,m]$ are for a given $n \ge 1$ 
(this is a consequence of Theorem~\ref{ThNyb}),
and then use this to determine $\TAU(n)$ (Theorem~\ref{ThK}). 

The next theorem is essentially due to  Nyblom~\cite{Nyb01}.
We give a proof since we will use the argument in the proof of Theorem~\ref{ThK}.

\begin{thm}\label{ThNyb}
For a given integer $S \ge 1$, all pairs of nonnegative integers $m$, $k$
such that
\beql{EqNyb1}
S ~=~ T_m - T_k
\eeq
are obtained in a unique way by factorizing $2S$ as a product $d \cdot e$ where $d$ is odd
and $e$ is even,
and taking
\begin{align}
k & ~=~ \frac{ \max(d,e)-\min(d,e)-1}{2} \, , \label{EqTT2a}\\
m & ~=~ \frac{ \max(d,e)+\min(d,e)-1}{2} \,. \label{EqTT2b}
\end{align}
\end{thm}
\begin{proof}
From \eqref{EqNyb1} we have
$$
2S = m(m+1)-k(k+1) = (m-k)(m+k+1)\,.
$$
Since their sum is odd, $m-k$ and $m+k+1$ are of opposite parity,
and also $m-k<m+k+1$.
Let $d$ be whichever of $m-k$ and $m+k+1$ is odd,
and let $e$ be the other.
Then $m-k= \min(d,e)$, $m+k+1=\max(d,e)$, and
solving for $k$ and $m$ we get \eqref{EqTT2a}, \eqref{EqTT2b}.
The uniqueness follows since conversely $k$ and $m$ determine $d$ and $e$.
\end{proof}

In particular, as Nyblom~\cite{Nyb01} shows, the number of pairs $(m,k)$ such that~\eqref{EqNyb1} holds
is equal to the number of odd divisors of $2S$.

We now take $S = T_n$.  Theorem~\ref{ThNyb} gives all triangular triples $[n,k,m]$ containing $n$.
There are always two obvious factorizations, $2T_n = 1 \cdot n(n+1)$
with $d=1$ and $e=n(n+1)$,
and $2T_n = n\cdot (n+1)$, with  $\{d,e\} = \{n, n+1\}$. The first case leads
to the triple $[n,T_n-1,T_n]$ already mentioned,
and the second leads to the trivial solution $[n,0,n]$.
It follows that the number of nontrivial triangular triples for a 
given $n$ (see A309507) is equal to the number of
odd divisors $d>1$ of $2T_n$.

This result is reminiscent of the fact that the number of primitive Pythagorean triples
with an even  leg $2uv$ is equal to the number
of odd divisors of $2uv$ (cf. \cite{Tri08}, A024361).
The nontrivial triangular triples $[n,k,m]$ sorted into lexicographic order are given by
$$
[n, A333530(n), A333531(n)]\,, 
$$
or by
$[n, A198455(n), A198456(n)]$ if we impose the restriction that $k \ge n$.
(Lee and Zafrullah~\cite{LeZa93} also give some tables of triangular triples.)
The numbers $n$ such that there is a triple $[n,n,m]$ are listed in A053141.

The following property and its elegant proof are due to Bradley Klee (personal communication).

\begin{thm}\label{ThTT2}
If $[n,k,m]$ is a triangular triple, then
\beql{EqTT3}
n+k ~\ge ~m\,.
\eeq
Equality holds if and only if $n=0$ or $k=0$.
\end{thm}
\begin{proof}
If we set $x=2n+1$, $y=2k+1$, $z=2m+1$ then \eqref{EqTT1} becomes
$$
x^2+y^2 ~=~ z^2+1\,.
$$
Certainly $[x,y,z]$ is not (quite) a Pythagorean triple, but this equation does suggest using the 
triangle inequality, which yields
$$
x+y ~\ge~ \sqrt{z^2+1} ~>~ z\,,
$$
and so 
$$
n+k ~>~ m-\frac{1}{2}\,,
$$
and \eqref{EqTT3} follows since all the quantities are integers.
If equality holds in \eqref{EqTT3} then $n^2+k^2=m^2$ (from \eqref{EqTT1})
and so $kn=0$.
\end{proof}

We can now apply Theorem~\ref{ThNyb} to determine $\TAU(n)$.

\begin{thm}\label{ThK}
For $n \ge 2$, $\TAU(n)$ is obtained by choosing that odd divisor $d$ of $n(n+1)$ which is 
different from $n$ and $n+1$, and minimizes
\beql{Eqtau1}
\left | d ~-~ \frac{n(n+1)}{d} \right |\,.
\eeq
Then $\TAU(n)$ is is the value of $k$ given by \eqref{EqTT2a} with
this value of $d$ and  $e=n(n+1)/d$. 
\end{thm}
\begin{proof}
From \eqref{EqTT2a} we see that the minimal $k$ is obtained by choosing
$d$ and $e$ so as to minimize $\max(d,e)-\min(d,e)$. But $d$ and $e$ are constrained by $d\cdot e = n(n+1)$.
So we must minimize \eqref{Eqtau1}.
 Since we require $k>0$, we must avoid $d=n$ and $d=n+1$.
\end{proof}

\paragraph{\it Remark.} In a few cases there is no need to do any minimization. For if $n$ is a Mersenne prime, or if $n+1$ 
is a Fermat prime, then the only odd divisor of $n(n+1)$ apart from $n$ or $n+1$
is $d=1$, and we get $\TAU(n) = T_n-1$.

\

We now return to our study of $\{A(n)\}$, and explain the connection with triangular triples.
The agreement of $q_n$ and $\TAU(n-1)$ is no coincidence.

\begin{thm}\label{ThAT}
For $n \ge 3$, $q_n = \TAU(n-1)$.
\end{thm}
\begin{proof}
Note that  $R$ is a triangular number if and
only if $8R+1$ is the square of an odd integer.
Indeed, $8T_n+1 = 4n^2 + 4n + 1 = (2n+1)^2$.
The proof of the theorem  is in two parts.

\noindent (i) Given $n \ge 3$, let $k$ denote the smallest nonnegative
integer such that $d = n+k+1$ divides
$$
p ~=~ (k+1)n+\frac{k(k+1)}{2}\,.
$$
Then 
\beql{EqAT1}
q ~=~ \frac{(k+1)n+\frac{k(k+1)}{2}}{n+k+1}
\eeq
is such that $R = T_{n-1} + T_q$ is a triangular number.
Indeed,
$8R+1 = (\alpha /d)^2$, where
$$
\alpha = 2n^2 + 2kn + k^2 + n + 2k + 1 = (2n+2k+1)d-2p\,,
$$
which is certainly divisible by $d$.
(These calculations were 
performed in Maple, but they
can easily be verified by hand.)
This proves that $\TAU(n-1) \leq q_n$.

(ii)  Conversely, suppose $n \ge 3$ and $q = \TAU(n-1)$ is such that 
\beql{EqAT2}
T_{n-1}+T_q ~=~ T_m
\eeq
for some integer $m$.
For given values of $n$ and $q$, \eqref{EqAT1} is a quadratic equation for $k$,
 and the unique solution with $k \ge 0$ is
$$
k ~=~ -n+q-\frac{1}{2} + \frac{1}{2} \sqrt{4n^2+4q^2-4n+4q+1}\,.
$$
Using \eqref{EqAT2} we can rewrite this as
$$
k ~=~ q+m-n\,,
$$
from which we get
$$
p ~=~ (k+1)n+\frac{k(k+1)}{2} ~=~ \frac{(q+m+n)(q+m-n+1)}{2} ~=~ q \, (n+k+1).
$$
This proves that $q_n \leq \TAU(n-1)$.
\end{proof}

In row $n$ of Table~\ref{TabA}, $A(n)$ corresponds to
 to the triangular triple $[n-1, q_n, m]$, where $m$ is given in the final column.

To summarize: initially we found $A(n)$ by seeing when a certain series of
trial divisions finally succeeded. Our analysis shows that an explicit answer is 
given by first finding $q_n$ from Theorems~\ref{ThK} and \ref{ThAT},
 finding $m$ by solving the quadratic equation $T_{n-1} + T_{q_n} = T_m$,
and then $A(n) = q_n + m - n$.  For example, if $n=5$, we find that $q_5= \TAU(4) = 9$, and
$T_4+T_9 = 10+45 = 55 = T_m$ tells us that $m=10$ and $A(5) = 9+10-5 = 14$.


\section{The multiplicative version, $\{B(n)\}$.}\label{SecB}
For the multiplicative version we replace the addition
of $n+i$ in the definition of $a_n(i)$ by multiplication, keeping the stopping rule.
So we define $B(n)$ for $n \ge 1$ by introducing an intermediate sequence $\{b_n(i), i \ge 0\}$ which 
starts with $b_n(0)=n$, and, for $i \ge 1$, satisfies $b_n(i)=b_n(i-1)\cdot (n+i)$. 
We stop when we  reach a term $b_n(k)$ which is divisible by $d=n+k+1$, and set $B(n)=k$.  
In other words, if the number $d$ that we are about to multiply $b_n(k)$ by actually divides $b_n(k)$, 
then instead of multiplying by it we stop.

An equivalent definition is that $B(n)$ is the smallest positive integer $k = k(n)$ such that $d(n) = n+k+1$ divides
\beql{EqB1}
b_n(k)  ~=~ \frac{(n+k)!}{(n-1)!}.
\eeq

When $n=1$, for example, the sequence $\{b_1(i)\}$ is $1,2,6,24,120$, and
we stop with $k=4=B(1)$ since $120$ is divisible by $d=1+4+1=6$.
For $n=4$, the sequence $\{b_4(i)\}$ is $4,20,120,840$, and we stop with $k=3=B(4)$ since $840$ is divisible by $d=4+3+1=8$.

Table \ref{TabB} gives the values of $B(n)=k(n)$, $d(n)=n+k(n)+1$, $p(n) = b_n(k)$, and
$q_n=p(n)/d(n)$ for
$n=1,2,\ldots,12$.

\begin{table}[htb]
\caption{}\label{TabB}
$$
\begin{array}{rcrrr}
n & B(n)=k(n) & d(n) & p(n)=b_n(k) & q_n=p(n)/d(n)  \\
\hline
1  & 4          & 6 & 120       & 20   \\   
2  & 3          & 6 & 120       & 20   \\  
3  & 2          & 6 & 60        & 10   \\   
4  & 3          & 8 & 840       & 105  \\    
5  & 4          & 10 & 15120    & 1512 \\     
6  & 5          & 12 & 332640   & 27720 \\     
7  & 4          & 12 & 55440    & 46 20 \\     
8  & 3          & 12 & 7920     & 660  \\    
9  & 5          & 15 & 2162160  & 144144  \\    
10 & 4          & 15 & 240240   & 16016   \\   
11 & 6          & 18 & 98017920 & 5445440 \\     
12 & 5          & 18 & 8910720  & 495040        
\end{array}
$$
\end{table}

The sequences $\{B(n)\}$, $\{d(n)\}$, $\{p(n)\}$, $\{q_n\}$ have now been added to \cite{OEIS}: $\{B(n)\}$
is A332558.  Just as in the additive version, 
there is a close match with an existing sequence in \cite{OEIS}.
If we add $1$ to the values of $k(n)$ we get $5,4,3,4,5,6,\ldots$,
which appears to match the entry for A061836,
although the definitions are different.
The older sequence, which we will denote by $\{\beta(n)\}$,  
has a more natural definition: $\beta(n)$ for $n \ge 0$ 
is defined to be the smallest integer $\kappa > 0$ such that $n+\kappa$ divides $\kappa!$.

\begin{thm}\label{ThB1}
For $n \ge 1$, $\beta(n) = B(n)+1$.
\end{thm}
\begin{proof}
By definition, $B(n)$ is the smallest $k>0$ such that
\beql{EqB2}
n+k+1 ~|~ n(n+1)(n+2)\cdots (n+k)\,,
\eeq
whereas $\beta(n)$ is the smallest $\kappa > 0$ such that
$$
n+\kappa ~|~ 1 \cdot  2 \cdot 3 \cdots \kappa \,,
$$
or, replacing $\kappa$ by $k+1$, the smallest $k$ such that
\beql{EqB3}
n+k+1 ~|~ 1 \cdot  2 \cdot 3 \cdots (k+1).
\eeq
The ratio of the right-hand sides of \eqref{EqB2} and \eqref{EqB3}
equals 
$\binom{n+k}{k+1}$
which is an integer, so
the right-hand side of \eqref{EqB3}
divides the right-hand side of \eqref{EqB2}. 
So the value of $k$ defined by \eqref{EqB2} is less than or equal to the value defined by \eqref{EqB3}.
To complete the proof, it is enough to show that if $n+k+1$ divides $n(n+1)(n+2)\cdots (n+k)$
then it divides $(k+1)!$.
But $n+k+1$ also divides $(\sigma + n)(\sigma+n+1)(\sigma+n+2)\cdots (\sigma+n+k)$
for any $\sigma$ that is a multiple of $n+k+1$. Taking $\sigma=-(n+k+1)$,
that expression becomes $(-1)^{k+1} (k+1)!$.
\end{proof}

We do not know of any simple formula for $B(n)$ in terms of $n$. The following is a weak upper bound,
which at least shows that $B(n)$ always exists.

\begin{thm}\label{ThB2}
For $n \ge 3$, $B(n) \le n-1$.
\end{thm}
\begin{proof}
Substituting $k=n-1$ in \eqref{EqB1}, we get $b_n(n-1) = (2n-1)!/(n-1)!$,
which is  divisible by $n+k+1=2n$ for $n \ge 3$.
\end{proof}

\subsection{Asymptotic growth of $B(n)$.}
We conjecture that as $n$ goes to infinity,
\begin{equation} \label{eq_k}
  B(n) = \exp \left( (c + o(1)) (\log n)^{1/2} (\log\log n)^{1/2} \right),
\end{equation}
with $c = 1/ \sqrt{2} =  0.7071\ldots$
In the rest of this section
we sketch some arguments that
 support the conjecture.\footnote{A proof of Equation \eqref{eq_k} might be possible using the techniques of~\cite{LPP93}.} 

Since $B(n) = \beta(n) - 1$ from Theorem~\ref{ThB1}, we study the asymptotic
growth of $\beta(n)$ instead.
Let $\beta'(n)$ be the smallest integer $k \geq 1$ such that
$n+k$ is $k$-smooth (i.e., it has only factors less than or equal to $k$).
Since $k!$ is $k$-smooth, clearly $\beta'(n) \leq \beta(n)$.
The converse is not always true: $\beta(2) = 4$ but $\beta'(2) = 2$
since $2 + 2$ is $2$-smooth.
However for large $n$ this phenomenon becomes increasingly rare.
For $10^8 \leq n < 2 \cdot 10^8$, only  5.7\% of the values of $n$ are such that
 $\beta'(n) < \beta(n)$, 
and for $10^9 \leq n < 2 \cdot 10^9$ the proportion
decreases to 4.2\%. 
Our first unproved assumption is that $\beta(n)$ and $\beta'(n)$ have the
same asymptotic behavior, so that it suffices to
study the asymptotic behavior of $\beta'(n)$.

The number $\Psi(n,k)$ of $k$-smooth numbers  $\le n$ is given by
Dickman's $\rho$ function:
\[ \frac{\Psi(n,k)}{n} \approx \rho(u), \]
where $u = \log n/\log k$ \cite{Dic30, Gra08}.
As $u$ goes to infinity, we have \cite[Eq. (1.6)]{Gra08}:
\beql{EqG1}
\rho(u) = u^{-u + o(u)}\,.
\eeq
However, what we want is the \emph{local density} $\Psi'(n,k)$
around $n$. This is studied
in Kruppa's Ph.~D.~thesis \cite[formula (5.6)]{Kruppa}, where it is shown that
\beql{EqG2}
 \frac{\Psi'(n,k)}{n} ~\approx~ \rho(u) - \gamma \, \frac{\rho(u-1)}{\log n}\,,
 \eeq
 $\gamma$ being the Euler-Mascheroni constant.
In our case the local density is close to the global density.
For example for $n=10^{25}+2554$ we have $\beta'(n) = 29972$,
thus $u \approx 5.584$, which yields
$\rho(u) \approx 6.7 \cdot 10^{-5}$, and
$\gamma \rho(u-1)/\log n \approx 1.1 \cdot 10^{-5}$.
Our second unproved assumption is that the local density is
$\approx \rho(u)$ asymptotically.
The expected distance between two $k$-smooth numbers around $n$
being $\approx 1/\rho(u)$,
the expected distance between
a random $n$ and the next $k$-smooth number is thus $\approx 1/(2\rho(u))$.

The above arguments, together with \eqref{EqG1}, combine to suggest that $B(n)$ is approximately
equal to the solution $k$ of the equations
\beql{EqG5}
k ~\approx~ \frac{u^u}{2}, \quad u ~=~ \frac{\log n}{\log k}\,.
\eeq
We have
\begin{align}
\log k & ~\approx~ u \log u ~\approx~ \frac{\log n}{\log k}\,(\log\log n -\log\log k) \,, \nonumber \\
(\log k)^2 & ~\approx~ \log n \,(\log\log n - \log\log k)\,, \nonumber \\
2 \log\log k & ~\approx~ \ \log\log n\,, \nonumber
\end{align}
and so
$$
(\log k)^2 ~\approx~ \frac{1}{2} \log n \log\log n\,,
$$
which gives \eqref{eq_k}.

\begin{figure}[!ht]
        \centerline{\includegraphics[ width=3in]{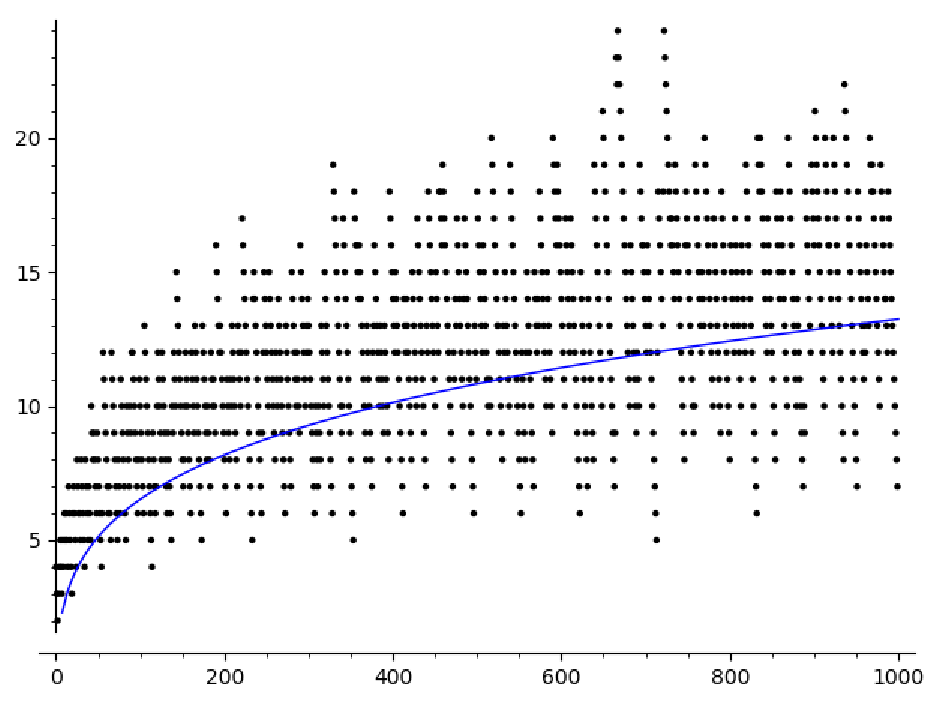}}
        \caption{The first $1000$ terms of $B(n)$ and (the curved line) $\overline{B(n)}$. }
        \label{Fig1}
        \end{figure}

Let 
\beql{EqGBar}
\overline{B(n)} ~:=~   \exp \left( \frac{1}{\sqrt{2}}\, (\log n)^{1/2} (\log\log n)^{1/2} \right)
\eeq
denote the main term on the right-hand side of \eqref{eq_k}.
$\overline{B(n)}$ is a reasonably good fit to $B(n)$, even for small $n$.
The graph in Fig.~\ref{Fig1} shows the first $1000$ terms of $B(n)$ and
(the curved line) $\overline{B(n)}$. 
We find that  $\overline{B(n)}$ is still a reasonably good fit to $B(n)$, even 
out to $n = 10^{30}$. Furthermore, it appears  that  
$\varlimsup_{n \to \infty}  B(n)$ is also  given by the right-hand side of \eqref{eq_k}.
It would be nice to know more about the asymptotic behavior of $B(n)$.


\section{The concatenation version, $\{C(n)\}$.}\label{SecC}

For the third version we replace addition and multiplication by concatenation,  but again keep the same stopping rule.
We define $C(n)$ for $n \ge 1$ by introducing an intermediate sequence $\{c_n(i), i \ge 0\}$ which
starts with $c_n(0)=n$, and, for $i \ge 1$, satisfies $c_n(i)=c_n(i-1) \| (n+i)$,
where $r \| s$ denotes the number whose decimal expansion is the concatenation of 
the decimal expansions of $r$ and $s$.
We stop if and when we  reach a term $c_n(k)$ which is divisible by $d=n+k+1$, and set $C(n)=k$.
In other words, if the number $d$ that we are about to concatenate to $c_n(k)$  actually divides $c_n(k)$,
then instead of concatenating $d$ we stop.
In contrast to the first two versions, here we do not have a proof that such a $k$ always exists.
It is theoretically possible that the sequence $c_n(i)$ never stops, in which case we define $C(n)$ to be $-1$.

 \begin{table}[htb]
\caption{}\label{TabC}
$$
\begin{array}{|rr|rr|rr|rr|}
\hline
 n & C(n) & n & C(n) & n & C(n) & n & C(n) \\
 \hline
1 & 1 & 26 & 33172 & 51 & 2249 & 76 & 320 \\
2 & 80 & 27 & 9 & 52 & 21326 & 77 & 59 \\
3 & 1885 & 28 & 14 & 53 & 53 & 78 & 248 \\
4 & 6838 & 29 & 317 & 54 & 98 & 79 & 31511 \\
5 & 1 & 30 & 708 & 55 & 43 & 80 & 20 \\
6 & 44 & 31 & 1501 & 56 & 20 & 81 & 5 \\
7 & 13 & 32 & 214 & 57 & 71 & 82 & 220 \\
8 & 2 & 33 & 37 & 58 & 218 & 83 & 49 \\
9 & 1311 & 34 & 34 & 59 & 91 & 84 & 12 \\
10 & 18 & 35 & 67 & 60 & 1282 & 85 & 25 \\
11 & 197 & 36 & 270 & 61 & 277 & 86 & 22 \\
12 & 20 & 37 & 19 & 62 & 56 & 87 & 105 \\
13 & 53 & 38 & 20188 & 63 & 47 & 88 & 34 \\
14 & 134 & 39 & 78277 & 64 & 106 & 89 & 4151 \\
15 & 993 & 40 & 10738 & 65 & 1 & 90 & 1648 \\
16 & 44 & 41 & 287 & 66 & 890 & 91 & 2221 \\
17 & 175 & 42 & 2390 & 67 & 75 & 92 & 218128159460 \\
18 & 124518 & 43 & 695 & 68 & 280 & 93 & 13 \\
19 & 263 & 44 & 2783191412912 & 69 & 19619 & 94 & 376 \\
20 & 26 & 45 & 3 & 70 & 148 & 95 & 23965 \\
21 & 107 & 46 & 700 & 71 & 15077 & 96 & 234 \\  
22 & 10 & 47 & 8303 & 72 & 64 & 97 & 321 \\
23 & 5 & 48 & 350 & 73 & 313 & 98 & 259110640 \\
24 & 62 & 49 & 21 & 74 & 34 & 99 & 109 \\
25 & 15 & 50 & 100 & 75 & 557 & 100 & 346 \\
\hline
\end{array}
$$
\end{table}

When $n=1$, for example, the sequence $\{c_1(i)\}$ is $c_1(0)=1$, $c_1(1)=1 \| 2 = 12$, and we stop with 
$k = 1 = C(1)$ since $12$ is divisible by $d=1+1+1=3$. 
Note that we cannot have $C(n) = 0$ since it would imply that $n+1$ divides
$n$.

For $n=7$, the sequence $\{c_7(i)\}$ is 
$$
7, 78, 789, 78910, 7891011, \ldots, 7891011121314151617181920\,,
$$
and after concatenating $20$ we stop with $k = 13 = C(7)$, since the last number there, which is $c_7(13)$,  is a multiple of $21$.

For $n=2$ the sequence $\{c_2(i)\}$ is 
$2,23,234, \ldots$ and stops with $k=80=C(2)$  at the $154$-digit number
$$
c_2(80) ~=~ 234567891011121314151617181920\ldots6970717273747576777879808182\,,
$$
which is divisible by $83$. 

Although a purist may be unhappy because its definition involves 
base $10$ arithmetic,\footnote{There is also a base-$2$ version,  with similar properties, although we will not
discuss it here: see A332563.} we find
$\{C(n)\}$ more interesting than $\{A(n)\}$ and $\{B(n)\}$ because 
its behavior is so erratic for such a simple rule, and 
we have no theoretical explanation for  this mixture of very small and very large numbers.

Table \ref{TabC} gives the values of $C(n)$ for $n \le 100$.
The values up to about $2 \cdot 10^5$ were found by straightforward direct search, 
but for the larger values we used the sieving algorithm described  
in the rest of this section.
At the present time we have found the exact value of $C(n)$ for all $n \le 1000$
except for two cases: $n=539$, where we only have upper and lower bounds, and $n=158$,
where we have searched up to $10^{14}$ without success,
and it is possible that $c_{158}(i)$ does not terminate.
The entry for $C(n)$ in \cite{OEIS}, A332580, includes a table  
for $n \le 1000$.
Although we do not have a proof that the sequence $\{c_n(i)\}$ always
terminates, the following heuristic argument suggests that it should.
After $k$ steps, we test $c_n(k)$ for divisibility by
$d=n+k+1$. There are three obvious cases when the division
is impossible: (i) when $n+k+1$ is even, since $c_n(k) \equiv n+k \pmod{10}$ is
odd and cannot be divisible by an even number;
(ii) when $n+k+1$ is a multiple of $5$, since then  $c_n(k) \equiv 4 \mbox{~or~} 9  \pmod{10}$
cannot be divisible by~$5$;
(iii) when $n+k+1$ is a multiple of $3$
in the case when $n \equiv 2 \pmod{3}$, since  then $n+k \equiv  2 \pmod{3}$
and so
$$c_n(k) \equiv 2 + 0 + 1 + 2 + \cdots + 0 + 1 + 2 \equiv 2 \pmod{3}.$$
Apart from this, $c_n(k)$ is essentially
a very large random number.\footnote{If $k$ has $j$ digits, $c_n(k)$
has about $jk$ digits, and we routinely search for $k$ up to $10^{11}$.}
The chance that $c_n(k)$ is divisible by $d$ is roughly $1/d$, and since
for a fixed $n$ the sum  $\sum_{k=1}^{\infty} 1/(n+k+1)$ diverges, we expect one of the divisions to succeed.
However,  we must admit that even when  we  try to make this argument more precise  by taking into account
conditions (i), (ii), (iii), the results do not fully explain the extreme irregularities in the   
values of $C(n)$ that can be seen in Table~\ref{TabC}. This sequence is still very mysterious.


\subsection{The concatenated words ${}_nW_m$.}\label{Sec4.1} 

We now present the  sieving algorithm  for $C(n)$ which we used
to obtain
\begin{align}\label{EqBigCs}
& C(44) = 2783191412912,~
C(92) = 218128159460, ~
C(494) = 2314160375788, \nonumber \\
& 10^{14} < C(539) \leq 887969738466613,~
C(761) = 615431116799,  \nonumber \\ 
& C(854) = 440578095296,   \mbox{and~} C(944) = 1032422879252.  
\end{align}

The numbers involved in this search are quite large\footnote{We
are tempted to say mind-boggling.}. To find $C(539)$, for example,
we must test numbers with about $10^{16}$ digits, that is,
numbers on the order of $10^{10^{16}}$, so see if they are divisible by numbers 
like $887969738466613$. Our algorithm (see~\S\ref{Sec4.3}) is therefore
fairly complicated, and requires some considerable technical machinery, which is developed in this section.


Although we only use it here for base $10$ calculations, we present the algorithm in
terms of an arbitrary base $b \ge 2$.
For an  integer $m \ge 0$, let $[m]_b$ denote its $b$ representation (formed by digits from $\{0,1,\dots,b-1\}$) starting with the most significant digit.

For any positive integer $m$, let $W_m$ be the integer
whose base-$b$ representation is the  concatenation of the base-$b$ representations of integers $1,\ 2,\ \dots,\ m$, that is
$$
[W_m]_b = [1]_b\,\|\,[2]_b\,\|\,\cdots\,\|\,[m]_b.
$$
We set $W_0:=0$.

Similarly, for integers $m\geq n\geq 1$, we define an integer ${}_nW_m$ by
$$
[{}_nW_m]_b = [n]_b\,\|\,[n+1]_b\,\|\,\cdots\,\|\,[m]_b,
$$
so that $W_m = {}_1W_m$. 

It can be seen that (extending the definition of  $c_n(i)$ to  base $b$), we have $c_n(i) = {}_nW_{n+i}$.
Correspondingly, the value of $C(n)$ (in base $b$) is given by $m-n-1$, 
where $m$ is the smallest positive integer such that $m>n$ and
\beql{eq:congm}
{}_nW_{m-1}\equiv 0\pmod{m}.
\eeq

For a positive integer $w$ we denote by $|w|$ the length of $[w]_b$. Clearly
$$
{}_nW_m = W_m - W_{n-1}\cdot b^{|{}_nW_m|}.
$$


In the following, we assume that $m$ is an $\ell$-digit integer,
with $\ell \geq 1$.
\begin{lemma}\label{lem:wlen}
Suppose $m$ is an $\ell$-digit integer in base $b$ (i.e.,
  $b^{\ell-1}\leq m< b^{\ell}$) and $1\leq n < m$.
Then 
$$|W_m| = \ell\cdot (m+1) - \frac{b^\ell-1}{b-1}$$
and
$$|{}_nW_m| = |W_m| - |W_{n-1}|.$$
In particular,
$$|W_{b^\ell-1}| = \ell\cdot b^\ell - \frac{b^\ell-1}{b-1}$$
and for any nonnegative integer $k<\ell$,
\begin{equation} \label{eq:lemma41_4}
|{}_{b^{k}}W_m| = \ell\cdot (m+1) - \left(\frac{b^{\ell-k}-1}{b-1} + k\right)\cdot b^k.
\end{equation}
\end{lemma}

\begin{proof}
It is easy to see that $[W_m]_b$ is formed by the concatenation of $b^k - b^{k-1}$ $k$-digit numbers for each $k=1,2,\dots,\ell-1$, and $m-b^{\ell-1}+1$ $\ell$-digits numbers.
Hence,
$$|W_m| = \sum_{k=1}^{\ell-1} k\cdot (b^k-b^{k-1}) + \ell\cdot (m-b^{\ell-1}+1) = \ell\cdot (m+1) - \frac{b^\ell-1}{b-1}.$$
Since $[{}_nW_m]_b$ is obtained from $[W_m]_b$ by removing the prefix $[W_{n-1}]_b$, we have
$$|{}_nW_m| = |W_m| - |W_{n-1}|.$$
\end{proof}


\begin{lemma}\label{lem:wround}
Suppose that $b^{\ell-1}\leq m< b^{\ell}$. Then
$$
{}_{b^{\ell-1}}W_m = \frac{m - (m + 1)b^\ell + (b^{2\ell-1} - b^{\ell-1} + 1) b^{\ell(m+1-b^{\ell-1})}}{(b^\ell-1)^2}.
$$
In particular, for $m=b^\ell-1$ we have
$$
{}_{b^{\ell-1}}W_{b^\ell-1} = \frac{b^\ell-1 - b^{2\ell} + (b^{2\ell-1} - b^{\ell-1} + 1) b^{\ell(b^\ell-b^{\ell-1})}}{(b^\ell-1)^2}.
$$
\end{lemma}

\begin{proof}
Notice that 
$$[{}_{b^{\ell-1}}W_m]_b = [b^{\ell-1}]_b\,\|\,[b^{\ell-1}+1]_b\,\|\,\dots\,\|\,[m]_b,$$
where each term in the right hand side is composed of $\ell$ digits. It follows that
$$
{}_{b^{\ell-1}}W_m = \sum_{i=0}^{m-b^{\ell-1}} (m-i)\cdot b^{\ell\cdot i} = \frac{m - (m + 1)b^\ell + (b^{2\ell-1} - b^{\ell-1} + 1) b^{\ell(m+1-b^{\ell-1})}}{(b^\ell-1)^2}.
$$
\end{proof}


\begin{lemma}\label{lem:wgen}
Suppose that $b^{\ell-1}\leq m< b^{\ell}$ and $1\leq n\leq m$. Then
$$
W_m = {}_{b^{\ell-1}}W_m + \sum_{k=1}^{\ell-1} {}_{b^{k-1}}W_{b^k-1}\cdot b^{ |{}_{b^k}W_m| } 
$$
and
$$
{}_nW_m = {}_{b^{\ell-1}}W_m - W_{n-1}\cdot b^{|{}_nW_m|} + \sum_{k=1}^{\ell-1} {}_{b^{k-1}}W_{b^k-1}\cdot b^{ |{}_{b^k}W_m| }.
$$
\end{lemma}

\begin{proof}
We notice that
$$[W_m]_b = [{}_{b^0}W_{b-1}]_b \,\|\, [{}_{b}W_{b^2-1}]_b \,\|\, \dots \,\|\, [{}_{b^{\ell-2}}W_{b^{\ell-1}-1}] \,\|\, [{}_{b^{\ell-1}}W_m]_b,$$
implying that
$$W_m = {}_{b^{\ell-1}}W_m + \sum_{k=1}^{\ell-1} {}_{b^{k-1}}W_{b^k-1}\cdot b^{|{}_{b^k}W_m|}.$$
By substituting   this expression into ${}_nW_m = W_m - W_{n-1}\cdot b^{|{}_nW_m|}$ we obtain the formula for ${}_nW_m$.
\end{proof}


\subsection{Prime powers dividing solutions to the congruence ${}_nW_{m-1}\equiv 0\pmod{m}$}\label{Sec4.2}

To find $C(n)$ we must find the smallest positive integer $m > n$ satisfying the congruence \eqref{eq:congm}.
As we will see in Section \ref{Sec4.3}, we  build $m$ from the set of prime powers dividing $m$. 
In this section we will show how to identify these prime powers.

Suppose that a solution $m$ to the congruence \eqref{eq:congm} is divisible by
a prime power $p^d$ for some integer $d \geq 1$.
Then $m$ is a solution to the pair of congruences:
\begin{equation}\label{eq:congq}
\begin{cases}
_nW_{m-1}\equiv 0\pmod{p^d}, \\
m\equiv 0\pmod{p^d}.
\end{cases}
\end{equation}
Our algorithm relies on the ability to identify the $\ell$-digit solutions $m$ to this system for any given $\ell$ and prime power $p^d$.
We will use  the following expression for $_nW_{m-1}$.

\begin{theorem}\label{th:abc} Suppose that $b^{\ell-1}\leq m-1 < b^{\ell}$ and $1\leq n < m$.
Then 
\begin{equation}\label{eq:wexp}
{}_nW_{m-1} = \frac{a_{n,\ell}\cdot b^{\ell (m-b^\ell)} - (b^\ell-1)m - 1}{(b^\ell-1)^2},
\end{equation}
where 
\[
\begin{split}
a_{n,\ell} & := 
(b^{2\ell-1} - b^{\ell-1} + 1)\cdot b^{\ell(b^\ell-b^{\ell  -1})} \\
& - {(b^\ell-1)^2} W_{n-1}\cdot b^{\ell b^{\ell} - \frac{b^\ell-1}{b-1} - |W_{n-1}|  } \\
& + {(b^\ell-1)^2} \sum_{k=1}^{\ell-1} {}_{b^{k-1}}W_{b^k-1}\cdot b^{\ell b^{\ell} - \left(\frac{b^{\ell-k}-1}{b-1} + k\right)\cdot b^k }
\end{split}
\]
is an integer that depends only on $n$ and $\ell$ but not on $m$. 
\end{theorem}

\begin{proof}
The  formulas for ${}_nW_{m-1}$ and $a_{n,\ell}$ follow from Lemmas~\ref{lem:wlen}-\ref{lem:wgen}. 
Noticing that $\ell b^{\ell} - \frac{b^\ell-1}{b-1} - |W_{n-1}| = |{}_nW_{b^{\ell}-1}|$ and 
$\ell b^{\ell} - \left(\frac{b^{\ell-k}-1}{b-1} + k\right)\cdot b^k  = |{}_{b^k}W_{b^{\ell}-1}|$ by Lemma~\ref{lem:wlen}, 
it is easy to verify that all exponents of $b$ in the formula for $a_{n,\ell}$ are nonnegative, and thus $a_{n,\ell}$ is an integer.
\end{proof}

\paragraph{\it Remark.} Since $a_{n,\ell}$ grows doubly exponentially in $\ell$, we may not be able to compute it explicitly even for relatively small values of $\ell$. 
However, we can efficiently compute $a_{n,\ell}$ modulo $q$ for a given positive integer $q$ as explained in Section~\ref{Sec4.3}.

\

It is important to note that the system~\eqref{eq:congq} has no solution when $p\mid b$. Indeed, if we assume that $m$ is a solution to the system~\eqref{eq:congq} for prime $p\mid b$ and $d=1$, 
then the number ${}_nW_{m-1} - (m-1) = {}_nW_{m-2}\cdot b^{\ell}$ is divisible by $b$ but not by $p$, a contradiction. 

Since $p\nmid b$, Theorem~\ref{th:abc} allows us to rewrite the system~\eqref{eq:congq} in the following equivalent form:
\begin{equation}\label{eq:congps}
\begin{cases}
a_{n,\ell}\cdot b^{\ell (m-b^\ell)} - (b^\ell-1)m - 1 \equiv 0\pmod{p^{d+2s}}, \\
m \equiv 0\pmod{p^d},
\end{cases}
\end{equation}
where $s:=\nu_p(b^{\ell}-1)$. From now on, we assume that the integers $n$, $\ell$, and a prime $p$ are fixed, and we are solving the system \eqref{eq:congps} with respect to $m$ for varying values of the integer $d\geq 1$.

Let $r$ be the multiplicative order of $b$ modulo $p$, and $r_1:=\frac{r}{\gcd(r,\ell)}$  the multiplicative order of $b^\ell$ modulo $p$. 
More generally, let $r_d$ be the multiplicative order of $b^\ell$ modulo $p^d$. It is clear that $r_d$ divides $r_1\cdot p^{d-1}$. 
In fact, $r_d = r_1\cdot p^{d-1-\delta}$ with $\delta\geq 0$, where $\delta=0$ unless $b^{\ell r_1}\equiv 1\pmod{p^2}$. In the latter case $p$ is a generalized Wieferich prime in base $b^\ell$ (and
such primes  are expected to be very rare).

We will consider two cases depending on whether $r_1=1$ (i.e., $r\mid \ell$) or $r_1 > 1$.

\subsubsection{Case $p \nmid b^{\ell}-1$}\label{Sec4.2.1} 

In this case, $r_1 > 1$, i.e., $s=0$ in the system \eqref{eq:congps}.
It follows that the system \eqref{eq:congps} has no solution when $\nu_p(a_{n,\ell})>0$. 
So we assume that $\nu_p(a_{n,\ell})=0$ and rewrite the system \eqref{eq:congps} in the form:
\begin{equation}\label{eq:conge}
\begin{cases}
b^{\ell m} \equiv \frac{b^{\ell b^\ell}}{a_{n,\ell}} \pmod{p^d}, \\
m\equiv 0\pmod{p^d}.
\end{cases}
\end{equation}

The following lemma enables us to lift the solutions to the system \eqref{eq:conge}
from $d=1$ to solutions for $d>1$.

\begin{lemma}\label{lem:lift} Let $k\geq 1$ be an integer.
If there is a solution $m$ to the system \eqref{eq:conge} for $d=k$, it has the form $m\equiv p^kt\pmod{p^kr_1}$ for some $t\in \mathbb{Z}_{r_1} = \{0,\dots,r_1-1\}$. 
Moreover, when $k>1$, $m\equiv p^{k-1}t\pmod{p^{k-1}r_1}$ is a solution to the system \eqref{eq:conge} for $d=k-1$.
\end{lemma}

\begin{proof}
If a solution $m$ to the first congruence of the system \eqref{eq:conge} for $d=k$ exists, 
it is given by a discrete logarithm to base $b^{\ell}$ modulo $p^k$, and thus it represents a residue modulo $r_k$.
At the same time, a solution to the second congruence of \eqref{eq:conge} represents the zero residue modulo $p^k$. 
Since $r_k = p^{k-1-\delta} r_1$ for some integer $\delta\geq 0$ and $\gcd(r_1,p)=1$, we have $\mathrm{lcm}(r_k,p^k)=p^k r_1$ and thus by the Chinese Remainder Theorem, 
the solutions to \eqref{eq:conge} are given by $m\equiv p^kt\pmod{p^kr_1}$ for some $t\in \mathbb{Z}_{r_1}$.

Now, if $m\equiv p^kt\pmod{p^kr_1}$ is a solution to \eqref{eq:conge} for $d=k$, then this $m$ also satisfies the system \eqref{eq:conge} for $d=k-1$. 
Since $p^kt\equiv p^{k-1}t\pmod{p^{k-1}r_1}$, we conclude that $m\equiv p^{k-1}t\pmod{p^{k-1}r_1}$ is a solution to the system \eqref{eq:conge} for $d=k-1$.
\end{proof}

The following theorem gives bounds on $d$ for the system \eqref{eq:conge} to be solvable, and describes the form of the solutions.


\begin{theorem}\label{th:lift} Suppose that the prime $p$ does not divide $b$ or $a_{n,\ell}$.
Let $m\equiv pt\pmod{pr_1}$ be a solution to the system \eqref{eq:conge} for $d=1$. 
Then the system \eqref{eq:conge} is solvable for a general $d$ if and only if
$1\leq d \leq D$, and the solutions are given by $m\equiv p^dt\pmod{p^dr_1}$, where
$$D:=\nu_p\left( \left(\frac{b^{\ell b^\ell}}{a_{n,\ell}}\right)^{r_1} - 1 \right) = \nu_p\left( a_{n,\ell}^{r_1} - b^{\ell b^\ell r_1} \right).$$
\end{theorem}

\begin{proof}
Suppose that $m=m_0$ is a solution to the system \eqref{eq:conge}. Then $p^d\mid m_0$, implying that $b^{\ell m_0 r_1}\equiv 1\pmod{p^d}$ since $r_d\mid p^{d-1}r_1\mid m_0 r_1$. 
Taking the first congruence of \eqref{eq:conge} to the power $r_1$, we get
$$a_{n,\ell}^{r_1} \equiv b^{\ell b^\ell r_1}\pmod{p^d},$$
implying that $d\leq D$.

Let $g$ be a primitive root modulo $p^D$, and thus $g$ modulo $p^D$ has the order $(p-1)p^{D-1}$.
Since $\big(\frac{b^{\ell b^\ell}}{a_{n,\ell}}\big)^{r_1} \equiv 1\pmod{p^D}$, it follows that $\frac{b^{\ell b^\ell}}{a_{n,\ell}} \equiv g^{s (p-1)p^{D-1}/r_1}\pmod{p^D}$ for some integer $s$, $0\leq s<r_1$.
Furthermore, since $sp^D\equiv sp^{D-1}\pmod{r_1p^{D-1}}$, we have 
\begin{equation}\label{eq:powp}
\left(\frac{b^{\ell b^\ell}}{a_{n,\ell}}\right)^p \equiv g^{ps (p-1)p^{D-1}/r_1} \equiv g^{s (p-1)p^{D-1}/r_1} \equiv \frac{b^{\ell b^\ell}}{a_{n,\ell}}\pmod{p^D}.
\end{equation}

We will show by induction on $d$ that $m\equiv p^dt\pmod{p^dr_1}$ is a solution to the system \eqref{eq:conge} for $d=1,2,\dots,D$. For $d=1$, this is given. If $m\equiv p^kt\pmod{p^kr_1}$ is a solution to the system \eqref{eq:conge} for $d=k<D$, then by taking the first congruence of \eqref{eq:conge} to the power $p$ and using the congruence \eqref{eq:powp}, we get
$$
b^{\ell p^{k+1}t} \equiv \left( \frac{b^{\ell b^\ell}}{a_{n,\ell}}\right)^p \equiv \frac{b^{\ell b^\ell}}{a_{n,\ell}}\pmod{p^{k+1}},
$$
which implies that $m\equiv p^{k+1}t\pmod{p^{k+1}r_1}$ is a solution to the system \eqref{eq:conge} for $d=k+1$ (since $k+1\leq D$).

By Lemma~\ref{lem:lift}, there exist no other solutions to the system \eqref{eq:conge} besides those constructed from $d=1$.
\end{proof}

Theorem~\ref{th:lift} allows us to concentrate on the case $d=1$. In this case, Theorem~\ref{th:lift} implies that $D\geq 1$ is necessary for the solubility of the system \eqref{eq:conge}, which is equivalent to
$$
a_{n,\ell}^{r_1} \equiv 1\pmod{p}.
$$
This condition holds trivially when $r_1=p-1$. However, it is nontrivial when $r_1<p-1$ and can 
be used as a quick test for solubility of the system \eqref{eq:conge}. If this condition holds, 
we proceed with computing the discrete logarithm of $a_{n,\ell}$ to  base $b^\ell$ modulo $p$. 
If the logarithm exists and equals $e$, i.e., $a_{n,\ell}\equiv b^{\ell e}\pmod{p}$, then from the first 
congruence of \eqref{eq:conge} it follows that $m\equiv b^\ell - e\pmod{r_1}$. 
Combining this with the second congruence of \eqref{eq:conge}, i.e., $m\equiv 0\pmod{p}$, 
we get a solution to the system \eqref{eq:conge} as $m\equiv p(b^\ell - e)\pmod{pr_1}$,
which we can lift using Lemma~\ref{lem:lift}.


\subsubsection{Case $p\mid b^{\ell}-1$.}  

In this case, we have $r_1=1$, i.e., $s>0$ in the system \eqref{eq:congps}.
We will need the following lemma.


\begin{lemma}\label{lem:binom} Let $m,u,d,s$ be positive integers and let $p$ be a prime such that $p^d\mid m$ and $p^s\mid u$. 
Then
$$(1+u)^m\equiv 1 + m u + \frac{m(m-1)}2 u^2 \pmod{p^{d+2s}},$$
which in the case $p\geq 3$ can be shortened to
$$(1+u)^m\equiv 1 + m u \pmod{p^{d+2s}}.$$
\end{lemma}

\begin{proof} Since 
$$(1+u)^m=1+m u+\frac{m(m-1)}2 u^2 + \sum_{k=3}^m \binom{m}k u^k,$$
it is enough to prove that $p^{d+2s}\mid \binom{m}{k}u^k$
for all $k$ in the range $3\leq k\leq m$. 
Consider two cases depending on whether or not $p$ divides $k$.

If $p\nmid k$, then 
$$p^{d+2s}\ \mid\ p^d u^2 \ \mid\ \frac{m}{k}\binom{m-1}{k-1}u^k = \binom{m}{k}u^k.$$

If $p\mid k$, then letting $t:=\nu_p(k)>0$, we get 
$$p^{d-t + ks}\ \mid\ \frac{m}{k}\binom{m-1}{k-1}u^k.$$
It remains to show that $d-t + ks\geq d + 2s$, i.e., $(k-2)s - t\geq 0$.
Except for the case when $p=2$ and $t=1$, $k\geq p^t$ implies
$$(k-2)s - t \geq p^t - 2 - t \geq 0.$$
On the other hand, if $p=2$ and $t=1$, from $k\geq 3$ it follows that 
$$(k-2)s - t \geq s-1 \geq 0.$$
\end{proof}

\begin{theorem} Let $p \geq 3$ be a prime such that $s:=\nu_p(b^{\ell}-1) > 0$. 
Then the system \eqref{eq:congps} has a solution if and only if $d\leq D$, where
$$D := \nu_p(a_{n,\ell} - b^{\ell b^\ell}) - 2s,$$
in which case the solutions are given by $m\equiv 0\pmod{p^d}$.
\end{theorem}

\begin{proof}
If $m$ is a solution to the system \eqref{eq:congps}, then $m\equiv 0\pmod {p^d}$, which by Lemma~\ref{lem:binom} with $u := b^{\ell}-1$ implies that
$$b^{\ell m} \equiv (1+b^{\ell}-1)^m \equiv 1 + m(b^\ell - 1) \pmod{p^{d+2s}},$$
and thus we can rewrite the first congruence of system \eqref{eq:congps} as
$$\left(\frac{a_{n,\ell}}{b^{\ell b^\ell}} - 1\right)(1 + m(b^\ell - 1)) \equiv 0 \pmod{p^{d+2s}}.$$
Since the second factor is coprime to $p$ and $p\nmid b$, we conclude that $m\equiv 0\pmod {p^d}$ gives a solution to the system \eqref{eq:congps} if and only if
$$a_{n,\ell}-b^{\ell b^\ell} \equiv 0 \pmod{p^{d+2s}},$$
which concludes the proof.
\end{proof}
\begin{theorem} Let $s:=\nu_2(b^{\ell}-1) > 0$. 
Then the system \eqref{eq:congps} for $p=2$ has a solution if and only if $d\leq D$, where
$$D := \nu_2(a_{n,\ell} - b^{\ell b^\ell}) - 2s + 1,$$
in which case the solutions are given by $m\equiv 0\pmod{2^{d+1}}$ if $d<D$, and by $m\equiv 2^d\pmod{2^{d+1}}$ if $d=D$.
\end{theorem}

\begin{proof}
First we notice that $s>0$ implies that $b$ is odd. 

Let $m$ be a solution to the system \eqref{eq:congps} for $p=2$. In particular, we have $\nu_2(m)\geq d$. 
Since by Lemma~\ref{lem:binom}
$$b^{\ell m} \equiv (1+b^{\ell}-1)^m \equiv 1 + m(b^\ell - 1) + \frac{m(m-1)(b^\ell - 1)^2}2 \pmod{2^{d+2s}},$$
the first congruence of system \eqref{eq:congps} for $p=2$ is equivalent to
\begin{equation}\label{eq:p2d2s}
\left(\frac{a_{n,\ell}}{b^{\ell b^\ell}}-1\right)(1 + m(b^\ell - 1)) + \frac{a_{n,\ell}}{b^{\ell b^\ell}}\frac{m(m-1)(b^\ell - 1)^2}2  \equiv 0 \pmod{2^{d+2s}}.
\end{equation}
It remains to consider two cases depending on whether $\nu_2(m)=d$ or $\nu_2(m)>d$:
\begin{itemize}
    \item If $\nu_2(m)=d$ (i.e., $m\equiv 2^d\pmod{2^{d+1}}$), then $m$ is a solution to the congruence \eqref{eq:p2d2s} if and only if $\nu_2(a_{n,\ell}-b^{\ell b^\ell})=d+2s-1$.
    \item If $\nu_2(m)>d$ (i.e., $m\equiv 0\pmod{2^{d+1}}$), then $m$ is a solution to the congruence \eqref{eq:p2d2s} if and only if $\nu_2(a_{n,\ell}-b^{\ell b^\ell})\geq d+2s$.
\end{itemize}
\end{proof}

\subsection{The $C(n)$ sieve.}\label{Sec4.3}
As we saw in \textsection\ref{Sec4.1}, to find $C(n)$ we must 
solve the congruence  \eqref{eq:congm}.
We construct solutions to this congruence using an analog of wheel 
factorization  \cite{Prit82}  for integers in the interval $[b^{\ell-1}, b^\ell)$. 
That is, we consider $p^d$ to be a factor of an integer $m$ only if $m$ is a solution to the system \eqref{eq:congps}. 
The integers $m$ that are factored completely in this way (i.e., the product of the identified 
factors equals $m$) give the solutions to the congruence \eqref{eq:congm}.
The following is \textsc{The $C(n)$ sieve}:
\

\begin{algorithmic}[1] 
  \Input an integer base $b\geq 2$, an integer $n$, and an upper bound $L := b^{\ell}$ for some $\ell\geq 1$
  \Output either $C(n)=m-n-1$ with $L/b < m \leq L$, or NONE if no such $m$ exists
  \State initialize an array $T[m] = 1$ for $m$ in the interval $(L/b,L]$        
  \For   {each prime $p \leq L$} \label{line2}                
   \State       $A\leftarrow \textsc{ComputeA}(n,\ell,b,p)$ \Comment{i.e., $A\leftarrow a_{n,\ell}\bmod p$}    \label{line:sd1}
   \If          {$p\mid b\quad\textbf{or}\quad A=0$}
   \State       proceed to the next value of $p$
   \EndIf
   \State       $r_1\leftarrow \mathrm{ord}_p(L)$ \Comment{i.e., multiplicative order of $L$ modulo $p$}
   \If             {$r_1>1$} 
   \If             {$A^{r_1}\not\equiv 1\pmod{p}$} 
   \State          proceed to the next value of $p$
   \EndIf
   \State       $q\leftarrow$ discrete logarithm of $A$ base $L$ modulo $p$; if it does not exist, continue to the next value of $p$
   \State       $M\leftarrow r_1\cdot p$
   \State       $m_0\leftarrow (L - q)p\,\bmod\,{M}$
   \Else   \Comment{case $r_1=1$}
     \State          $m_0\leftarrow 0$
     \If             {$p=2$} 
       \State          $M\leftarrow 4$
     \Else   
       \State          $M\leftarrow p$
     \EndIf
   \EndIf

     \If             {$m_0=0$}   \Comment{we consider only positive solutions}
       \State          $m_0\leftarrow M$
     \EndIf                                       \label{line:ed1}
   
   \State          $d\leftarrow 1$                \label{line:lss}

   \While          { $m_0\leq L$ \quad\textbf{and}\quad $(p,m_0,d)$ is a solution to the system \eqref{eq:congps} }
   
   \State          $m \leftarrow m_0 + \left\lceil\frac{L/b-m_0}{M}\right\rceil\cdot M $
   
   \While          {$m\leq L$}
  \State             $T[m] \leftarrow p \cdot T[m]$               
  \State             $m \leftarrow m + M$               
  \EndWhile
   \State          $m_0 \leftarrow m_0\cdot p$
   \State          $M \leftarrow M\cdot p$
   \State          $d \leftarrow d+1$
  \EndWhile                                       \label{line:lse}

  \If {$p=2$ \quad\textbf{and}\quad $m_0/2\leq L$ \quad\textbf{and}\quad $(p,m_0/2,d)$ is a solution to the system \eqref{eq:congps}}   \Comment{case $p=2$ and $d=D$}  \label{line:scs}
    \State $m_0\leftarrow m_0/2$
     \State          $m \leftarrow m_0 + \left\lceil\frac{L/b-m_0}{M}\right\rceil\cdot M $
   
     \While          {$m\leq L$}
      \State             $T[m] \leftarrow p \cdot T[m]$               
      \State             $m \leftarrow m + M$               
    \EndWhile
  \EndIf                                                                                               \label{line:sce}
 
  \EndFor

  \For    {$m$ from $\max\{L/b+1,n+2\}$ to $L$}       \label{line:coms}
  \If        {$T[m] = m$}       
  \State        return $C(n)=m-n-1$         \label{line:come}
  \EndIf
  \EndFor
  \State  return NONE             
\label{algo1}
\end{algorithmic}
\

\paragraph{\it Remarks.} 
Lines \ref{line:sd1}-\ref{line:ed1} find $m_0$ and $M$ such that the residues $m\equiv m_0\pmod{M}$ satisfy the system \eqref{eq:congq} for $d=1$.
In lines \ref{line:lss}-\ref{line:lse} we multiply $T[m]$ by the prime $p$  for solutions $m$ in the interval $(L/b,L]$, and incrementally lift the solutions to larger $d$. The exceptional case of $p=2$ and $d=D$ is addressed
in lines \ref{line:scs}-\ref{line:sce}. 
In lines \ref{line:coms}-\ref{line:come}, we check if any integer $m$ in the interval $(L/b,L]$ was factored completely, and derive $C(n)$ from the smallest such $m$.

\

Since the performance of our algorithm depends on our ability to compute $a_{n,\ell}\bmod p$,  we explain how to do this efficiently using Lemmas~\ref{lem:wlen}-\ref{lem:wgen}.

We start with a function based on Lemma~\ref{lem:wlen} that computes $|W_m|$.

\

\begin{algorithmic}[1] 
\Function{\sc LenW}{$m,b$} \Comment{computes $|W_m|$ in base $b$}
\If {$m=0$}
\State \Return 0
\EndIf
\State $\ell \leftarrow \left\lceil \log_b(m+1)\right\rceil$ \Comment{number of base-$b$ digits in $m$}
\State \Return $\ell\cdot (m+1) - (b^{\ell}-1)/(b-1)$
\EndFunction
\label{alg:lenw}
\end{algorithmic}

\

Next, we use Lemma~\ref{lem:wround} to design a function that for given integers $m,b,q$ with $b^{\ell-1}\leq m<b^\ell$ computes
${}_{b^{\ell-1}}W_m\bmod q$ in base $b$.

\

\begin{algorithmic}[1] 
\Function{\sc bWmQ}{$m,b,q$} \Comment{computes ${}_{b^{\ell-1}}W_m\bmod q$ in base $b$}
\State $\ell \leftarrow \left\lceil \log_b(m+1)\right\rceil$ \Comment{number of base-$b$ digits in $m$}
\State $d_1\leftarrow (b^{\ell}-1)^2$
\State $d_2\leftarrow 1$
\State $g\leftarrow \gcd(d_1,q)$
\While {$g>1$}
\State $d_1\leftarrow d_1 / g$
\State $d_2\leftarrow d_2 \cdot g$
\State $g \leftarrow \gcd(d_1,g)$
\EndWhile
\State $B\leftarrow b\pmod{q\cdot d_2}$ \Comment{a residue modulo $q\cdot d_2$}
\State \Return $\textsc{lift}\left((m - (m+1)\cdot B^{\ell} + (B^{2\ell-1} - B^{\ell-1}+1)\cdot B^{\ell(m+1-b^{\ell-1})}) / d_1\right)\ /\ d_2$
\EndFunction
\label{alg:valWbm}
\end{algorithmic}

\

The function {\sc bWmQ} first represents the denominator of the expression for ${}_{b^{\ell-1}}W_m$ from Lemma~\ref{lem:wround} as $(b^{\ell}-1)^2 = d_1\cdot d_2$, where $d_1$ and $d_2$ are the largest divisors of $(b^{\ell}-1)^2$ 
such that $d_1$ is co-prime to $q$ while $d_2$ is composed of prime factors dividing $q$.
Then the function defines $B$ to be a residue modulo $qd_2$ such that all computations involving $B$ are performed modulo the same number.
Namely, this function computes the numerator of the expression for ${}_{b^{\ell-1}}W_m$ divided by $d_1$ as a residue modulo $qd_2$, which is then lifted (with the  function \textsc{lift}) to an integer and divided by $d_2$.
This approach  produces the correct value for ${}_{b^{\ell-1}}W_m\bmod q$ even if $q$ is not co-prime to $b^{\ell}-1$ (i.e., when $d_2>1$).

Similarly, we use Lemma~\ref{lem:wgen} and the expression~\eqref{eq:lemma41_4} to design a recursive function that computes ${}_nW_m\bmod q$ for given integers $n,m,b,q$.

\

\begin{algorithmic}[1] 
\Function{\sc nWmQ}{$n,m,b,q$} \Comment{computes ${}_nW_m\bmod q$ in base $b$}
\State $\ell \leftarrow \left\lceil \log_b(m+1)\right\rceil$ \Comment{number of base-$b$ digits in $m$}
\State $B\leftarrow b\pmod{q}$ \Comment{a residue modulo $q$}
\State $r \leftarrow \textsc{bWmQ}(m,b,q) + \sum_{k=1}^{\ell-1} \textsc{bWmQ}(b^k-1,b,q)\cdot B^{\ell\cdot(m+1) - ((b^{\ell-k}-1)/(b-1)+k)\cdot b^k}$
\If {$n>1$}
\State $r \leftarrow r - \textsc{nWmQ}(1,n-1,b,q)\cdot B^{\textsc{LenW}(m,b)-\textsc{LenW}(n-1,b)}$
\EndIf
\State \Return $\textsc{lift}(r)$
\EndFunction
\label{alg:valWnm}
\end{algorithmic}

\

Finally, we are ready to design a function that computes $a_{n,\ell}\bmod q$ for given integers $n,\ell,b,q$. This function implements the formula:
$$a_{n,\ell} = (_nW_{b^\ell - 1} + 1) (b^\ell - 1)^2 + b^\ell,$$
which follows from Theorem~\ref{eq:wexp} for $m=b^\ell$.

\

\begin{algorithmic}[1] 
\Function{\sc ComputeA}{$n,\ell,b,q$} \Comment{computes $a_{n,\ell}\bmod q$ in base $b$}
\State $C\leftarrow b^\ell \pmod{q}$ \Comment{a residue modulo $q$}
\State \Return $\textsc{lift}\left( (\textsc{nWmQ}(n,b^\ell-1,b,q)+1) \cdot (C-1)^2 + C \right)$
\EndFunction
\label{alg:ComputeA}
\end{algorithmic}

\paragraph{\it Remark.} Although we use {\sc ComputeA} (and thus {\sc nWmQ} and {\sc bWmQ}) in our algorithm for computing $C(n)$
only for prime $q=p$, it works equally well for non-prime $q$.

\subsection{Discussion.} When the value $C(n)$ is small, less than $10^7$ (say), it can be computed directly by explicitly constructing $c_n(i)$ and testing its divisibility by $n+i+1$ for each $i=1,2,\dots$. 
In base 10, this naive search is faster than 
the $C(n)$ sieve
when $C(n)$ is below $10^7$ or so. 
However, for larger values of $C(n)$ the sieve gives a significant speed-up. 
It therefore makes sense to combine the two algorithms, by first running the naive search up to a certain threshold,   and then, if it was unsuccessful, switching to the sieve.
The choice for the threshold will depend on how the algorithms are implemented.\footnote{Our implementations are currently  available from \url{https://github.com/maxale/Recaman_cousin_C}}.


 \section*{Acknowledgments}
 We thank Michael~J.~Collins, Bradley Klee, Victor~S.~Miller, Kerry Mitchell,
 and Allan~C.~Wechsler
 for helpful comments during our work on $A(n)$,
 David~A.~Corneth, R\'{e}my Sigrist,  and
Jinyuan Wang for computing further terms in certain sequences arising in our study of $B(n)$,
 and Pierrick Gaudry for his help with the asymptotics
 of $B(n)$ and for  independently checking some of the calculations for $C(n)$.

\medskip

\noindent MSC2010: 11B83 (11D72, 11D85)


\begin{thebibliography}{99}

\bibitem{BaWe79} 
D.~W.~Ballew and R.~C.~Weber,
\emph{Pythagorean triples and triangular numbers},
The Fibonacci Quarterly, \textbf{17.2} (1979), 168--172.

\bibitem{CRSW} 
B.~Chaffin, E.~M.~Rains, N.~J.~A.~Sloane, and A.~R.~Wilks,
Numerical investigations of Recam\'an's sequence, in preparation, 2021.


\bibitem{Dic30}
K.~Dickman, 
\emph{On the frequency of numbers containing prime factors of a certain relative magnitude}, 
 Ark. Mat. Astr. Fys., \textbf{22.10} (1930), 1--14.

\bibitem{Gra08}
A.~Granville, 
\emph{Smooth numbers: computational number theory and beyond}, in
{\em  Algorithmic Number Theory: Lattices, Number Fields, Curves and Cryptography}, 
 Math. Sci. Res. Inst. Publ., \textbf{44}, Cambridge Univ. Press, Cambridge, 2008, 267--323.

\bibitem{Hag97} 
P.~W.~Haggard, 
\emph{Pythagorean triples and sums of triangular numbers},
Internat.~J. Mathematical Education in Science and Technology,
\textbf{28.1} (1997), 109--116.

\bibitem{Kruppa}
A.~Kruppa,
\emph{Speeding up integer multiplication and factorization},
PhD Dissertation, Univ. Henri Poincar\'e Nancy~1, 2010;
\url{http://docnum.univ-lorraine.fr/public/SCD_T_2010_0054_KRUPPA.pdf}.

\bibitem{LeZa93} 
H.~Lee and M.~Zafrullah, 
\emph{A note on triangular number},
 Punjab Univ.~J. Math., \textbf{26} (1993), 75--83.

\bibitem{LPP93}
H.~W.~Lenstra~Jr, J.~Pila, and C. ~Pomerance, 
\emph{A hyperelliptic smoothness test.~I},
 Phil. Trans. Roy. Soc. London, Series \textbf{A},
 \textbf{345.1676} (1993), 397--408.
 

\bibitem{Nyb01}
M.~A.~Nyblom, 
\emph{On the representation of the integers as a difference of nonconsecutive triangular numbers},
The Fibonacci Quarterly, \textbf{39.3} (2001), 256--263.

\bibitem{OEIS}  
OEIS Foundation Inc.~(2021),
{\em The On-Line Encyclopedia of Integer Sequences},
\url{https://oeis.org}.

\bibitem{Prit82}
P.~Pritchard,
\emph{Explaining the wheel sieve},
Acta Informatica, \textbf{17} (1982), 477--485.

\bibitem{Sie63} 
W.~Sierpi\'{n}ski,  
\emph{On triangular numbers which are sums of two smaller triangular numbers  [Polish]}, 
Wiadom. Mat., 
\textbf{(2) 7} (1963): 27--28; MR0182602.

\bibitem{Tri08} 
A.~Tripathi, 
\emph{On Pythagorean triples containing a fixed integer}, 
The Fibonacci Quarterly, \textbf{46/47.4} (2008/09), 331--340. 

\bibitem{Ula09} 
M.~Ulas,
\emph{A note on Sierpi\'{n}ski's problem related to triangular numbers},
Colloq. Math., \textbf{117.2} (2009): 165--173.

\bibitem{Vai72} 
A.~M.~Vaidya, 
\emph{On representing an integer as a sum of two triangular numbers},
Vidya, \textbf{B 15.2} (1972), 104--105.

\end{thebibliography}
\end{document}